\documentclass[10pt]{amsart}
\usepackage[letterpaper, margin=1.3in]{geometry}
\usepackage{graphicx}
\usepackage{amsmath,amssymb,amsthm,mathtools}
\usepackage[all]{xy}
\usepackage[utf8]{inputenc}
\usepackage{tikz}
\usepackage{tikz-cd}
\usepackage[hidelinks]{hyperref}

\newtheorem{theorem}{Theorem}[section]

\newtheorem{lemma}[theorem]{Lemma}
\newtheorem{remark}[theorem]{Remark}
\newtheorem{definition}[theorem]{Definition}
\newtheorem{proposition}[theorem]{Proposition}
\newtheorem{example}[theorem]{Example}

\newtheorem{question}[theorem]{Question}
\newtheorem*{theorem*}{Theorem}
\begin{document}

\title{Brauer class over the Picard scheme of totally degenerate stable curves}

\date{\today}

\author{Qixiao Ma}
\address{Shanghai Center for Mathematical Sciences, Fudan University}
\email{qxma10@fudan.edu.cn}

\begin{abstract}
   We study the Brauer class rising from the obstruction to the existence of tautological line bundles on the Picard scheme of curves. We determine the period and index of the Brauer class in certain cases.
\end{abstract}

\maketitle

\tableofcontents

\section{Introduction}
Let's work over a fixed base field $k_0$, e.g. $\mathbb{C}$. Let $\Gamma$ be a finite graph, with degree at least $4$ at each vertex. Let $X/k$ be the universal totally degenerate stable curve with dual graph $\Gamma$. Let $\mathrm{Pic}^0_{X/k}$ be the identity component of the Picard scheme of $X$. Tautological line bundles do not necessarily exist over $X\times\mathrm{Pic}^0_{X/k}$, the obstruction is given by a Brauer class $\alpha\in\mathrm{Br}(\mathrm{Pic}^0_{X/k})$. We determine the period and index of the Brauer class in certain cases. The motivation of this work is to produce some division algebras of prime degree, serving as candidates of the cyclicity problem, see \cite{AESU}.

In section \ref{sec3}, we associate the Brauer class with some extension class analogous to the ``elementary obstruction'' in \cite{Sko}. In section \ref{sec4}, we use this observation to get lower bounds on period of the Brauer class. The most interesting bound is:

\begin{theorem*}
Let $\sigma$ be an element in $\mathrm{Aut}(\Gamma)$ of order $m$.  If $L$ is a simple loop in $\Gamma$ fixed by $\sigma$, such that $\langle L\rangle \subseteq \mathrm{H}_1(\Gamma,\mathbb{Z})$ is a direct summand as $\langle\sigma\rangle$-module, then $m$ divides $\mathrm{per}(\alpha)$.
\end{theorem*}

In section \ref{sec6}, we give a list of examples. In section \ref{sec5} we discuss some further questions.
\vspace{1mm}

\textbf{Acknowledgements.} This work is a part of my thesis at Columbia University. I am very grateful to my advisor Aise Johan de Jong for his invaluable ideas, enlightening discussions and unceasing encouragement. I am also very grateful to the 2018 Brauer Group conference, where I first had chance to present this work and got warmly encouraged by many people.

\section{Preliminaries}\label{sec2}

\subsection{Torus and its character module}\label{2.1} We recall some basic properties of tori, see \cite{Mil}.

Let $k$ be a field, let $k^s$ be its separable closure. Let $G=\mathrm{Gal}(k^s/k)$ be the absolute Galois group of $k$.
Let $X$ be a scheme over $k$. For any $\sigma\in G$, we denote the morphism $1_X\times\mathrm{Spec}(\sigma)\colon X_{k^s}\to X_{k^s}$ by $\sigma_X$.

Let $T$ be a group scheme over $k$. We say $T$ is a torus if $T_{k^s}\cong (\mathbb{G}_{m,k^s})^r$. The character module of $T$ is a right $\mathbb{Z}[G]$-module, denoted by $\mathrm{X}(T)$, whose underlining abelian group is $\mathrm{Hom}(T_{k^{s}},\mathbb{G}_{m,k^{s}})\cong\mathbb{Z}^r$. The $G$-action on $\mathrm{X}(T)$ is given by $(\sigma,\chi)\mapsto \sigma(\chi)=\sigma_{\mathbb{G}_m}\circ\chi\circ\sigma_T^{-1}$.
We say the torus $T$ is split if $T\cong\mathbb{G}_{m,k}^r$.
The contravariant functor $T\mapsto \mathrm{X}(T)$ induces an equivalence between the category of $k$-tori and the category of $G$-representation on finite rank free $\mathbb{Z}$-modules.

\subsection{Totally degenerate stable curves}
Let $k$ be a field, let $k^s$ be its separable closure. Let $X$ be a $k$-scheme, such that $X_{k^s}$ is a Deligne-Mumford stable curve. We say $X$ is a totally degenerate stable curve, if all the irreducible components of $X_{k^s}$ are rational curves. The dual graph of $X$, is a graph whose vertices are labelled by irreducible components of $X_{k^s}$, for each node in $X_{k^s}$, we assign an edge at corresponding vertices. Let $\Gamma$ be the dual graph of $X$, let's denote $\mathrm{dim}_{\mathbb{Q}}\mathrm{H}_1(\Gamma,\mathbb{Q})$ by $g(\Gamma)$, it coincides with the arithmetic genus of $X$.

\subsection{Moduli of totally degenerate stable curves}\label{2.3}
Let $\Gamma=(V,E)$ be a graph, with $|E(v)|\geq 4$ for each $v\in V$. Let's work over a fixed base field $k_0$. Let $\mathcal{M}_{g(\Gamma)}$ be the moduli stack of stable genus $g(\Gamma)$ curves.
It has a locally closed substack $\mathcal{D}_{\Gamma}$, parameterizing families of totally degenerate nodal curve with dual graph $\Gamma$. For any finite set $S$, let $\mathcal{M}_{0,S}$ be the moduli stack of stable genus $0$ curves with distinct $S$-labeled marked points. This is an irreducible smooth $k_0$ variety.
Let $\mathcal{M}_{\Gamma}=\prod_{v\in V}\mathcal{M}_{0,E(v)}$. There is an obvious $\mathrm{Aut}(\Gamma)$-action on $\mathcal{M}_\Gamma$.
The $\mathcal{D}_\Gamma$ can be naturally identified with $[\mathcal{M}_{\Gamma}/\mathrm{Aut}(\Gamma)]$ as the image of the clutching morphism $c\colon\mathcal{M}_\Gamma\to\mathcal{M}_{g(\Gamma)}$.

Since we assumed $|E(v)|\geq4$ for the graph $\Gamma$, the $\mathrm{Aut}(\Gamma)$-action on $\mathcal{M}_\Gamma$ is generically free, hence $\mathcal{D}_\Gamma=[\mathcal{M}_\Gamma/\mathrm{Aut}(\Gamma)]$ has an open substack represented by a scheme $D_\Gamma^\circ$, and there exists a universal family of totally degenerate nodal curves with dual graph $\Gamma$ over $D_\Gamma^\circ$.
Taking generic fiber, we obtain a totally degenerate nodal curve $X$ over $k=\mathrm{frac}(D^\circ_\Gamma)$. Let $k'$ be the function field of $\mathcal{M}_\Gamma$, then $k'/k$ is a Galois extension with $\mathrm{Gal}(k'/k)=\mathrm{Aut}(\Gamma)$, and $X_{k'}$ is a union of $\mathbb{P}^1_{k'}$s with $k'$-rational nodes and dual graph $\Gamma$. We refer to \cite[XII.10]{ACGH} for details.

\subsection{Picard scheme of curves}\label{Picardofcurve}
We recall some facts about the Picard scheme of curves, for details, see \cite[8]{BLR}. Let's work over a fixed field $k$, let's denote its separable closure by $k^s$.

Let $X$ be a geometrically reduced and geometrically connected proper curve defined over $k$. Consider the functor $P'_{X/k}\colon \mathrm{Sch}/k\to \mathrm{Sets},\ T\mapsto\mathrm{Pic}(X\times T)/\mathrm{pr}_2^*\mathrm{Pic}(T)$. We define the Picard functor $P_{X/k}$ to be the \'etale sheafification of $P'_{X/k}$. We know $P'_{X/k}$ is in general not representable, but $P_{X/k}$ is always represented by a smooth $k$-group scheme $\mathrm{Pic}_{X/k}$. Note that the representability of $P'_{X/k}$ means there exist tautological line bundles on $X\times \mathrm{Pic}_{X/k}\to\mathrm{Pic}_{X/k}$, and
the representability of $P_{X/k}$ means there exist tautological line bundles after some \'etale base change $U\to\mathrm{Pic}_{X/k}$. If $k'/k$ is a separable extension such that $X(k')\neq0$, we can always take $U=\mathrm{Pic}_{X/k}\times_kk'$.

Let $\mathrm{Pic}^0_{X/k}$ be the identity component of $\mathrm{Pic}_{X/k}$. When $X$ be a totally degenerate nodal curve, one can show $\mathrm{Pic}^0_{X/k}$ is a torus. More precisely, let $\nu\colon X^\nu_{k^s}\to X_{k^s}$ be the normalization map, let $\pi\colon X_{k^s}\to \mathrm{Spec}(k^s)$ be the structure map. Apply $\mathrm{R}\pi_*$ to the short exact sequence:

$$\xymatrix{
0\ar[r]&\mathbb{G}_m\ar[r]&\nu_*\mathbb{G}_m\ar[r]&\nu_*\mathbb{G}_m/\mathbb{G}_m\ar[r]& 0},
$$
let's denote the dual graph of $X$ by $\Gamma=(V,E)$, then we have long exact sequence:

$$\xymatrix{
0\ar[r]&\mathbb{G}_{m,k^s}\ar[r]& \mathbb{G}_{m,k^s}^{\oplus|V|}\ar[r]&\mathbb{G}_{m,k^s}^{\oplus|E|}\ar[r]& \mathrm{Pic}_{X_{k^s}/k^s}\ar[r]^{\mathrm{deg}}\ar[r] &\mathbb{Z}^{\oplus|V|}\to 0},$$ here the map ``deg'' assigns a line bundle to the degree of its restriction on each component. Then $\mathrm{Pic}^0_{X_{k^s}/k^s}=\mathrm{ker}(\mathrm{deg})$ descends to the torus $\mathrm{Pic}^0_{X/k}$. Its functor of points represents line bundles whose pullback to the normalization $X_{k^s}^\nu$ has degree $0$ on each component, we call such line bundles multi-degree zero line bundles.

To sum up, descending the above exact sequence, we get an exact sequence of tori:
$$\xymatrix{
0\ar[r]&\mathbb{G}_m\ar[r]&T_1\ar[r]&T_2\ar[r]& \mathrm{Pic}^0_{X/k}\ar[r]&0}.
$$

Taking characters, we conclude that:
\begin{proposition}\label{homology}
If $X$ is a totally degenerate curve with dual graph $\Gamma$, then there is canonical $\mathrm{Gal}(k^s/k)$-module isomorphism $$\mathrm{X}(\mathrm{Pic}^0_{X/k})\cong \mathrm{H}_1(\Gamma,\mathbb{Z}).$$
\end{proposition}
\begin{proof} Apply $\mathrm{X}(-)$ to the above exact sequence, we get the simplicial chain complex that calculates the homology of $\Gamma$. It is easy to check this is an $\mathrm{Gal}(k^s/k)$-isomorphism.
\end{proof}
\section{The Brauer class}\label{sec3}
Let $\Gamma$ be a graph with degree at leat $4$ at each vertex. We work over a fixed field $k_0$. Let $\mathcal{D}_\Gamma$ be the moduli stack of totally degenerate curves over $k_0$ with dual graph $\Gamma$. Let $k=\mathrm{frac}(\mathcal{D}_\Gamma)$ be its function field, let $X/k$ be the universal curve. Let $k'/k$ be the $\mathrm{Aut}(\Gamma)$-Galois extension that splits $X$. Let's denote $\mathrm{Aut}(\Gamma)$ by $G$.

We introduce the Brauer class $\alpha\in\mathrm{Br}(\mathrm{Pic}_{X/k}^0)$ that will be studied later. We define an obstruction class $\alpha_d$ and an extension class $\alpha_e$. We show they are related by $j(\alpha_e)=\alpha_d,i(\alpha_d)=\alpha$ under natural injections $$\xymatrix{\mathrm{H}^2(G,\mathrm{X}(\mathrm{Pic}^0_{X/k}))\ar[r]^-{j}& \mathrm{H}^2(G,\mathrm{H}^0(\mathrm{Pic}^0_{X_{k'}/k'},\mathbb{G}_m))\ar[r]^-{i}& \mathrm{H}^2(\mathrm{Pic}^0_{X/k},\mathbb{G}_m).}$$

\subsection{Cohomological description}\label{coh}
Consider the projection $\pi\colon X\times\mathrm{Pic}^0_{X/k}\to\mathrm{Pic}^0_{X/k}$ and structure map $s\colon\mathrm{Pic}^0_{X/k}\to \mathrm{Spec}(k)$. The Leray spectral sequence $\mathrm{R}s_*\mathrm{R}\pi_*\mathbb{G}_m\Rightarrow\mathrm{R}(s\pi)_*\mathbb{G}_m$ implies the following low term short exact sequence: $$0\to\mathrm{Pic}(\mathrm{Pic}^0_{X/k})\to\mathrm{Pic}(X\times\mathrm{Pic}^0_{X/k})\to\mathrm{Pic}_{X/k}(\mathrm{Pic}^0_{X/k})\overset{\mathrm{d}_2^{0,1}}{\to}\mathrm{H}^2(\mathrm{Pic}^0_{X/k},\mathbb{G}_m),$$ see \cite[8.1.4]{BLR} for details. Then $\alpha=\mathrm{d}_2^{0,1}(1_{\mathrm{Pic}^0_{X/k}})$ is a cohomological Brauer class over $\mathrm{Pic}^0_{X/k}$. It is clear from above exact sequence that $\alpha$ is the obstruction to existence of a tautological line bundle on $X\times\mathrm{Pic}^0_{X/k}$.

\begin{definition}
We call $\alpha\in \mathrm{H}^2(\mathrm{Pic}^0_{X/k},\mathbb{G}_m)$ the Brauer class associated with the graph $\Gamma$.
\end{definition}
The following proposition shows the class lies in the Azumaya Brauer group.
\begin{proposition} Let $f\colon T\to S$ be a projective, flat and finitely presented morphism, with connected geometric fibers, and $S$ is quasi-compact, then $\mathrm{d}_{2}^{0,1}(\mathrm{Pic}_{T/S}(S))\subset \mathrm{Br}(S).$
\end{proposition}
\begin{proof} See \cite[4.9.1]{Giraud}.
\end{proof}

\subsection{Galois descent description}\label{Galde}
Note that $X(k')\neq \emptyset$, therefore tautological line bundles exist on $X\times\mathrm{Pic}^0_{X_{k'}/k'}$. Two tautological line bundles differ by tensoring pullback of a line bundle on $\mathrm{Pic}^0_{X_{k'}/k'}$.
Note that $\mathrm{Pic}^0_{X_{k'}/k'}\cong(\mathbb{G}_{m,k'})^{\oplus g(\Gamma)}$ is the affine scheme of a unique factorization domain, it has only trivial line bundles, therefore tautological line bundles on $X\times\mathrm{Pic}^0_{X_{k'}/k'}$ are unique up to isomorphism, denote it by $\mathcal{L}$.

Let's denote $X\times\mathrm{Pic}^0_{X/k}$ by $Y$. Let's denote the base change of $\sigma_Y$\footnote{See section \ref{2.1}} to $Y_{k'}$ still by $\sigma_Y$. Note that by uniqueness, tautological line bundles pull back to tautological lines bundle under base change, thus $\mathcal{L}$ and $\sigma_Y^*\mathcal{L}$ are isomorphic. For each $\sigma\in G$, let's fix an isomorphism $\phi_\sigma\colon \mathcal{L}\cong \sigma_Y^*\mathcal{L}$.
We show the collection of isomorphisms $\{\phi_\sigma\}_{\sigma\in G}$ can be modified to a descent datum if and only if some second order group cohomology class vanishes.

Let's note that choosing the isomorphisms $\phi_\sigma$ are equivalent to pinning down the image of $1$ in for some isomorphism $s_\sigma \colon\mathcal{O}\to \mathcal{L}^{-1}\otimes\sigma_Y^*\mathcal{L}$. We fix a rational section $t$ of $\mathcal{L}$, and define the isomorphisms $r_\sigma\colon\mathrm{pr}_{2,*}(\mathcal{L}^{-1}\otimes\sigma_Y^*\mathcal{L})\cong\mathcal{O}$ on $\mathrm{Pic}^0_{X_{k'}/k'}$ such that $r_\sigma(t^{-1}\otimes\sigma_Y^*t)=1$.

Consider the cocycle $\alpha_d\colon (\sigma,\tau) \mapsto r_{\sigma}(s_{\sigma}(1))\cdot (r_{\sigma\tau}(s_{\sigma\tau}(1)))^{-1}\cdot\sigma^*(r_{\tau}(s_{\tau}(1)))$,  unravelling the definitions, one checks that the collection of isomorphisms $\{\phi_{\sigma}\}_{\sigma\in G}$ can be modified to descent datum if and only if the class of $\alpha_d$ is zero in $\mathrm{H}^2(G,\mathrm{H}^0(\mathrm{Pic}^0_{X_{k'}/k'},\mathbb{G}_m))$. Since Galois descent of coherent sheaves is effective, the class $\alpha_d$ is the obstruction to existence to tautological line bundles on $X\times\mathrm{Pic}^0_{X/k}$.

\begin{definition}We call $\alpha_d\in \mathrm{H}^2(G,\mathrm{H}^0(\mathrm{Pic}^0_{X_{k'}/k'},\mathbb{G}_m))$ the descent obstruction class associated with $\Gamma$.
\end{definition}
By the Hochschild-Serre spectral sequence, we have natural edge map
$$\xymatrix{i\colon \mathrm{H}^2(G,\mathrm{H}^0(\mathrm{Pic}^0_{X_{k'}/k},\mathbb{G}_m))\ar[r]&  \mathrm{H}^2(\mathrm{Pic}^0_{X/k},\mathbb{G}_m).}$$
\begin{proposition}\label{iinj}The map $i$ is an injection.
\end{proposition}

\begin{proof}The low term short exact sequence
shows the kernel is generated by the image of $\mathrm{H}^1(G,\mathrm{Pic}^0(X_{k'}))$. But
 $\mathrm{H}^1(G,\mathrm{Pic}^0(X_{k'}))=\mathrm{H}^1(G,\{1\})=0$.
\end{proof}
\begin{proposition}\label{extclass1} We have $i(\alpha_d)=\alpha$.\end{proposition}
\begin{proof}
Consider the following array of functors
from the category of sheaves to category of modules: $$\xymatrix{
\mathrm{Sh}(X\times\mathrm{Pic}^0_{X_{k'}/k'})
\ar[r]^-{\mathrm{pr}_{2,*}}&
\mathrm{Sh}(\mathrm{Pic}^0_{X_{k'}/k'})
\ar[r]^-{\Gamma(-)}&
\mathrm{Mod}_{k[G]}
\ar[r]^-{{(-)^G}}
&\mathrm{Mod}_k,}$$
contracting the first two functors or last two functors, we get two spectral sequences. By \cite[08BI]{SP}, there exists natural morphism: $$\xymatrix{z\colon\mathrm{Pic}_{X/k}(\mathrm{Pic}^0_{X/k})\ar[r]& \mathrm{H}^0(G,\mathrm{Pic}_{X_{k'}/k'}(\mathrm{Pic}^0_{X_{k'}/k'}))\ar[d]^{\sim}&\\
&\mathrm{H}^0(G,\mathrm{H}^1(\mathrm{Pic}^0_{X_{k'}/k'},\mathrm{pr}_{2,*}\mathbb{G}_m))\ar[r]& \mathrm{H}^2(G,\mathrm{H}^0(\mathrm{Pic}^0_{X_{k'}/k'},\mathbb{G}_m))
}$$ such that $i\circ z=\mathrm{d}_2^{0,1}$. One checks that $z(1_{\mathrm{Pic}^0_{X/k}})=\alpha_d$, thus $i(\alpha_d)=\alpha$.
\end{proof}

\subsection{Extension class description}
Recall in section \ref{Picardofcurve}, we displayed an exact sequence of tori: $$0\to\mathbb{G}_m\to T_1\to T_2\to\mathrm{Pic}^0_{X/k}\to 0.$$ Taking character modules, we get $$0\to \mathrm{X}(\mathrm{Pic}^0_{X/k})\to \mathrm{X}(T_2)\to \mathrm{X}(T_1)\to \mathrm{X}(\mathbb{G}_{m})\to0,$$ The sequence of Galois modules defines an extension class $$\alpha_e\in\mathrm{Ext}^2_G(\mathbb{Z},\mathrm{X}(\mathrm{Pic}^0_{X/k}))=H^2(G,\mathrm{X}(\mathrm{Pic}^0_{X/k})).$$ \begin{definition}We call $\alpha_e\in \mathrm{H}^2(G,\mathrm{X}(\mathrm{Pic}^0_{X/k}))$ the extension class associated with $\Gamma$.
\end{definition}

Let's give explicit descriptions of the class $\alpha_e$.
Let $T_3$ be the image of $T_1$ in $T_2$, we have short exact sequences $$\xymatrix{0\ar[r] &\mathbb{G}_{m}\ar[r] &T_1\ar[r] &T_3\ar[r] &0} \eqno{(a)}$$ and
$$\xymatrix{0\ar[r] &T_3\ar[r] &T_2\ar[r]&\mathrm{Pic}^0_{X/k}\ar[r] &0.}\eqno{(b)}$$

Let $M$ be a $G$-module, let $C^\bullet(G,M)$ be the canonical resolution by inhomogeneous cochains $C^i(G,M)=\mathrm{Map}(G^i\to M)$. Let $\delta$ be the connecting homomorphism on cohomology groups.
\begin{proposition}\label{cal} The class $\alpha_e$ can be calculated in the following two ways.
\begin{enumerate}
\item Apply $C^\bullet(G,\mathrm{Hom}(-,\mathrm{X}(\mathrm{Pic}^0_{X/k})))$ to $(b),(a)$. Then $\alpha_e=\delta\circ\delta(1_{\mathrm{X}(\mathrm{Pic}^0_{X/k})}).$
\item Apply $C^\bullet(G,\mathrm{Hom}(\mathrm{X}(\mathbb{G}_m),-))$ to $(a),(b)$. Then $\alpha_e=\delta\circ\delta(1_{\mathrm{X}(\mathbb{G}_m)}).$
\end{enumerate}
\end{proposition}
\begin{proof} The Ext class can be calculated by either the first or second variable in $\mathrm{Hom}(-,-)$.\end{proof}

\begin{proposition}\label{ses}
We have a split short exact sequence $$0\to \mathrm{X}(\mathrm{Pic}^0_{X/k})\to H^0(\mathrm{Pic}^0_{X_{k'}/k'},\mathbb{G}_m){\to}\mathbb{G}_{m.k'}\to 0,$$where the first map is inclusion, and second map is evaluating at identity section $e\in\mathrm{Pic}^0_{X/k}$.
\end{proposition}
\begin{proof}Fix an isomorphism $\mathrm{Pic}^0_{X_{k'}/k'}\cong\mathbb{G}_{m,k'}^{g}\hookrightarrow\mathbb{A}^{g}_{k'}$. Suppose the coordinate on $\mathbb{A}_{k'}^{g}$ are given by $T_1,\dots, T_g$ such that $\mathrm{Pic}^0_{X_{k'}/k'}\cong\mathrm{Spec}(k'[T_1^{\pm1},\dots,T_g^{\pm1}])$.
Given any regular section of $s\in H^0(\mathrm{\mathrm{Pic}^0_{X_{k'}/k'}},\mathbb{G}_m)$ such that $s(e)=1$, let $v_i(s)$ be the valuation of $s$ along $T_i$. Then $s'=\prod T_i^{-v_i(s)}s$ is a rational section on $\mathbb{A}^g_{k'}$, regular and nonvanishing on the complement of codimension at least $2$ subset. By Hartogs's theorem, see \cite[Tag 031T]{SP}, we know $s'\in H^0(\mathbb{A}^2_{k'},\mathbb{G}_m)=\mathbb{G}_{m,k'}$. Then $s'(e)=\prod_{i}T_i(e)^{-v_i(s)}s(e)=1$ implies $s'=1$, hence $s=\prod_iT_i^{v_i(s)}$, thus the exactness in middle. The right and left exactness are obvious. The splitting map is given by pulling back scalars. \end{proof}

\begin{proposition}\label{jinj}The natural map induced by inclusion$$j\colon \mathrm{H}^2(G,\mathrm{X}(\mathrm{Pic}^0_{X/k}))\to \mathrm{H}^2(G,\mathrm{H}^0(\mathrm{Pic}^0_{X_{k'}/k'},\mathbb{G}_m))$$is an injection.
\end{proposition}
\begin{proof}This is inclusion of a direct summand.\end{proof}

\begin{proposition}\label{extclass2} We have $j(a_e)=a_d$.\end{proposition}
\begin{proof}Consider the following natural transformation of functors from category of tori to abelian groups, induced by the equivalence between the category of Galois modules and tori:
$$\xymatrix{\mathrm{Hom}_{G}(\mathrm{X}(-),\mathrm{X}(\mathrm{Pic}^0_{X/k}))\ar[r]& ((-_{k'})(\mathrm{Pic}^0_{X_{k'}/k'}))^G.}$$
The class $\alpha_e$ is calculated by iterated $\delta$-homomorphisms of $1_{\mathrm{X}(\mathrm{Pic}^0_{X/k})}$ on the left hand side. Let's denote $\mathrm{Pic}^0_{X_{k'}/k'}$ by $P$. Consider the short exact sequences of $G$-modules
$$\xymatrix{0\ar[r]&\mathbb{G}_{m}(P)\ar[r] &T_{1,{k'}}(P)\ar[r]& T_{3,{k'}}(P)\ar[r]&0}\eqno{(a')}$$ and
$$\xymatrix{0\ar[r]&T_{3,k'}(P)\ar[r]& T_{2,k'}(P)\ar[r]& P(P)\ar[r]& 0.}\eqno{(b')}$$

The natural transformation commutes with $\delta$-homomorphisms. Hence in order to show $i(\alpha_e)=\alpha_d$, it suffices to show the obstruction class $\alpha_d$ is also calculated by iterated $\delta$-homomorphism on $1_{P}$, for the $C^\bullet(G,-)$ resolutions. In order to describe $1_P$, we need explain how to describe line bundles on $X\times P$. More precisely, we describe the map $T_{2,k'}(P)\to P(P).$
Given a $2|E|$-tuple: $\{r_{v,e}\}_{v\in V,e\in E(v)}$, where $ r_{v,e}\in \Gamma(P,\mathbb{G}_m).$
Take the trivial line bundle on the normalization
$X^\nu\times P$ of $X\times P$. Fix a nowhere vanishing section $s$ for the trivial line bundle. In the dual graph $\Gamma$, for each edge $e$ with endpoints $v,w$, we denote the $k'$-section corresponding to $e\in E$ on component $v$ of $X_{k'}^\nu$ by $Q_{v,e}$. We identify the trivial line bundle over $Q_{v,e}\times P$ and $Q_{w,e}\times P$ by
$r_{v,e}\cdot s|_{Q_{v,e}\times P}=r_{w,e}\cdot s|_{Q_{w,e}\times P}.$
We get a multidegree zero line bundle on $X \times P$.
By the construction, we get same line bundle if we scale $r_{v,e},r_{w,e}$ by the same function. Hence we get the map $(\prod_{e\in E}\mathbb{G}_m\times\mathbb{G}_m/\mathbb{G}_m)(P)\cong T_{2,k'}(P)\to P(P)$. This map is surjection as we may read off the gluing datum on normalization.

Fix any datum $\{r_{v,e}\}$ that represents the tautological line bundle. Let $Z=X\times P$ and $Z^\nu=X^\nu\times P$. Recall $\delta$-homomorphism is defined as taking coboundary of an inverse image. We first note that calculating the first coboundary of the lifting $\{r_{e,v}\}$ is the same as pinning down isomorphisms $\{\phi_\sigma\}$:
An isomorphism $\phi_{\sigma}\colon \mathcal{L}\cong\sigma_{Z}^*\mathcal{L}$ is determined by an isomorphism on their pullbacks to $Z^\nu$. Since $\mathcal{L}$ is trivialized on $Z^\nu$, the isomorphism $\phi_\sigma|_{Z^\nu}$ is multiplication by $u_{v,\sigma}=\sigma_P^*(r_{\sigma^{-1}v,\sigma^{-1}e})/r_{v,e}\in \mathrm{H}^0(P,\mathbb{G}_m)$ on the $v$-component\footnote{This is independent of choice of $e\in E(v)$} of $Z^\nu$.
So the isomorphism is represented by $\{u_{v,\sigma}\}\in T_{1,k'}(P)$.
Taking a second coboundary, we get cocycle $(\sigma,\tau)\mapsto\phi_{\sigma}|_{Z^\nu}\cdot \phi_{\sigma\tau}|_{Z^\nu}^{-1}\cdot\sigma_{Z^\nu}^*(\phi_{\tau}|_{Z^\nu})$. This is multiplication by elements in $\Gamma(P,\mathbb{G}_m)$ on each component. Since it is the pullback of $\phi_\sigma\cdot\phi_{\sigma\tau}^{-1}\cdot\sigma_Z^*(\phi_\tau)$ from $Z$ and $Z$ is connected, the multiplication on each component are the by same element in $\Gamma(P,\mathbb{G}_m)$, denoted by $c_{\sigma,\tau}$. We note that $(\sigma,\tau)\mapsto c_{\sigma,\tau}$ is just the descent obstruction class $\alpha_d$ as defined in section \ref{Galde}.
\end{proof}

\section{Period and index}\label{sec4}
Let $S$ be a scheme, recall for a class $\beta\in\mathrm{Br}(S)$, the period of $\beta$ is the least $n\in\mathbb{Z}_{>0}$ such that $n\beta=0$. The index of $\beta$ is the greatest common divisor of rank of Azumaya algebras representing the class. In general period always divides index, see \cite[1.3]{antieau2014}. If $S$ is a regular scheme, let $\mathrm{frac}(S)$ be the function field of $S$, then we have injection $\mathrm{Br}(\mathrm{frac}(S))\to\mathrm{Br}(S)$, they share the same period and index, see \cite[6.1]{Ant}. In this case, we will not distinguish the Brauer class in $\mathrm{Br}(S)$ or in $\mathrm{Br}(\mathrm{frac}(S))$.

\subsection{Bounds on index}
Let $X$ be a stable curve over a field $k$, with genus at least $2$. As explained in section \ref{coh}, there exists a Brauer class $\alpha_X=\mathrm{d}_{2}^{0,1}(\mathrm{Pic}^0_{X_{k}/k})\in \mathrm{Br}(\mathrm{Pic}_{X/k}^0)$, which is the obstruction class to existence of a universal line bundle on $X\times \mathrm{Pic}^0_{X/k}$.
\begin{theorem}\label{ind} The index $\mathrm{ind}(\alpha_X)$ divides $g-1$.
\end{theorem}
\begin{proof}Note there is a canonical isomorphism $\mathrm{Pic}^0_{X/k}\cong\mathrm{Pic}^{2g-2}_{X/k}$
by tensoring with canonical line bundle. Let $S=\mathrm{Pic}^{2g-2}_{X/k}\backslash\{\omega_C\}$. By \cite[V.4]{Giraud}, up to a sign, the Brauer class $\alpha_X$ is represented by class of the Brauer-Severi scheme $\mathrm{Div}_{X/k}^{2g-2}$ over $S$, which has relative dimension $g-2$.
\end{proof}
For totally degenerate nodal curves, we have extra criterion. Let's work over a fixed field $k_0$. Let $\Gamma$ be graph, such that each vertex has degree at least $4$. Let $X/k$ be the universal curve. Let $\alpha$ be the Brauer class associated with graph $\Gamma$. Suppose $\Gamma_0$ is an $\mathrm{Aut}(\Gamma)$-invariant subgraph, with $v_0$ vertices and $e_0$ edges.
\begin{theorem}\label{extra1} The index $\mathrm{ind}(\alpha)$ divides $e_0$ and $2v_0$.
\end{theorem}
\begin{proof}The nodes corresponding to edges in $\Gamma_0$ descend to a union of closed points $P_1\cup\dots\cup P_s$ over $k$. Let $\kappa(P_i)$ be the fraction field of $P_i$. Since $X_{\kappa(P_i)}$ has a $\kappa(P_i)$-rational point, the universal line bundle exist over $X_{\kappa(P_i)}\times\mathrm{Pic}^0_{X_{\kappa(P_i)}/{\kappa(P_i)}}$, hence the Brauer class is zero in $\mathrm{Br}(\mathrm{Pic}^0_{X/k}\times_k\kappa(P_i))$. By \cite[4.5.8]{GS}, we know $\mathrm{ind}(\alpha)|[\kappa(P_i)\colon k]$ for each $i$, thus $\mathrm{ind}(\alpha)|e_0$.

The components corresponding to vertices in $\Gamma_0$ descend to an integral scheme $X_0$ over $k$. Let $X_0^\nu$ be the normalization of $X_0$, let $X_0^\nu\overset{a}{\to}S_0\overset{b}{\to}\mathrm{Spec}(k)$ be the Stein factorization, see \cite[Tag 03GX]{SP}, where $a$ is a conic bundle and $b$ is a union of closed points $Q_0\cup\dots\cup Q_r$. Since $X_0^\nu\to S_0$ is a conic bundle, for each $i$, $X_0^\nu$ has infinitely many degree $2\mathrm{deg}(Q_i)$ points over $k$, hence does $X_0$ and $X$. Then we argue as in the last theorem.
\end{proof}

\subsection{Bounds on period}\label{sec4}
Let $\Gamma$ be a graph, such that each vertex has degree at least $4$. Let $X/k$ be the universal curve associated with graph $\Gamma$. Let $G=\mathrm{Aut}(\Gamma)$.
Recall that $\mathrm{X}(\mathrm{Pic}^0_{X/k})$ can be identified with $\mathrm{H}_1(\Gamma,\mathbb{Z})$ as $G$-module, see section \ref{homology}.

Let $\sigma$ be a period $m$ element in $G$. Let $L$ be a loop in $\Gamma$ with no repeated vertices. Suppose $L$ is fixed by $\sigma$ and we can write $L=\cup_{i=0}^{m-1}\sigma^i(ve)$, where $v$ is a vertex in $L$ and $e$ is the union of edges connecting $v$ and $\sigma(v)$. Let $\langle\sigma\rangle\subset G$ be the subgroup generated by $\sigma$.

\begin{theorem}\label{thm} View $\mathrm{H}_1(\Gamma,\mathbb{Z})$ as a $\langle\sigma\rangle$-module. If the submodule generated by $L$ is a direct summand, then $m$ divides $\mathrm{per}(\alpha)$.
\end{theorem}
\begin{proof} Consider the restriction map $\mathrm{H}^2(G,\mathrm{X}(\mathrm{Pic}^0_{X/k}))\to \mathrm{H}^2(\langle\sigma\rangle,\mathrm{X}(\mathrm{Pic}^0_{X/k}))$.
Since $\mathbb{Z}\cdot L$ is a summand of $\mathrm{X}(\mathrm{Pic}_{X/k}^0)$, we have projection map $\mathrm{H}^2(\langle\sigma\rangle,\mathrm{X}(\mathrm{Pic}^0_{X/k}))\to H^2(\langle\sigma\rangle,\mathbb{Z})$. By Proposition \ref{iinj} and \ref{jinj}, we know $i\circ j\colon \mathrm{H}^2(G,\mathrm{X}(\mathrm{Pic}^0_{X/k}))\to \mathrm{H}^2(\mathrm{Pic}^0_{X/k},\mathbb{G}_m)$ is injection,
thus it suffices to show $\alpha_e$ maps to a generator in $H^2(\langle\sigma\rangle,\mathbb{Z})\cong\mathbb{Z}/m\mathbb{Z}.$

We use $\delta\circ\delta(1_{\mathrm{X}(\mathbb{G}_m)})$ to calculate $\alpha_e$, see Proposition \ref{cal}. An inverse image of $1\in \mathrm{X}(\mathbb{G}_m)$ in $\mathrm{X}(T_1)\cong\mathbb{Z}^{\oplus|V|}$ can be taken to be characteristic function $\chi_{v}$. Take differential we get $1$-cocycle $\phi\colon \sigma^i\mapsto \chi_{\sigma^i(v)}-\chi_{v}$ in $\mathrm{X}(T_3)$. We take an inverse image of $\phi$ in $\mathrm{X}(T_2)=\mathbb{Z}^{\oplus|E|}$. For example, $\sigma^i\mapsto \sigma^0(e)+\dots+\sigma^{i-1}(e)$. Taking differential again, the $2$-cocycle can be written as $(\sigma^i,\sigma^j)\mapsto c_{i,j}$, where $c_{i,j}=0$ if $i+j<n$, $c_{i,j}=1\cdot L$ if $i+j\geq n$. Thus image of $\alpha$ in $\mathrm{H}^2(\langle\sigma\rangle,\mathbb{Z})\cong\mathbb{Z}/m\mathbb{Z}$ is the canonical generator, hence $m|\mathrm{per}(\alpha)$.\end{proof}
\begin{theorem}\label{simpleest} The period $\mathrm{per}(\alpha)$ divides $|G|=|\mathrm{Aut}(\Gamma)|$.
\end{theorem}
\begin{proof} Note $\alpha$ is the image of $\alpha_e$, and $|\mathrm{Aut}(\Gamma)|$ annihilates $\alpha_e$ in group cohomology.
\end{proof}

\subsection{Functorial property}
Let $k$ be a field, let $f\colon X\to Y$ be a morphism of proper geometrically connected $k$-curves. Let $\phi\colon\mathrm{Pic}^0_{Y/k}\to\mathrm{Pic}^0_{X/k}$ be the natural morphism induced by pullback. Consider the following commutative diagram
\[\xymatrix{
  Y\times\mathrm{Pic}^0_{Y/k}\ar[d]^{p_1} &X\times\mathrm{Pic}^0_{Y/k} \ar[l]_{f\times1}\ar[r]^{1\times\phi}\ar[d]^{p_2}& X\times\mathrm{Pic}^0_{X/k}\ar[d]^{p_3} \\%
\mathrm{Pic}^0_{Y/k} & \mathrm{Pic}^0_{Y/k}\ar[r]^{\phi}\ar@{=}[l]& \mathrm{Pic}^0_{X/k}
}\]

Consider the Leray spectral sequence associated to the projections in the columns and the sheaf $\mathbb{G}_m$. By functoriality of Leray spectral sequence, we have the following commutative diagram, where the map $\mathrm{d}_{2}^{0,1}$s are denoted by $\delta_i$:

\[\xymatrix{
 \mathrm{Pic}^0_{Y/k}(\mathrm{Pic}^0_{Y/k})\ar[d]^{\delta_1}\ar[r]^{\phi\circ-} &\mathrm{Pic}^0_{X/k}(\mathrm{Pic}^0_{Y/k})\ar[d]^{\delta_2}& \mathrm{Pic}^0_{X/k}(\mathrm{Pic}^0_{X/k})\ar[d]^{\delta_3}\ar[l]_{-\circ\phi} \\%
\mathrm{H}^2(\mathrm{Pic}^0_{Y/k},\mathbb{G}_m)\ar[r]^{\sim} & \mathrm{H}^2(\mathrm{Pic}^0_{Y/k},\mathbb{G}_m) & \mathrm{H}^2(\mathrm{Pic}^0_{X/k},\mathbb{G}_m)\ar[l]_{\phi^*}
}\]

\begin{theorem}\label{thmpull} Let $\alpha_X=\delta_3(1_{\mathrm{Pic}^0_{X/k}})$, $\alpha_Y=\delta_1(1_{\mathrm{Pic}^0_{Y/k}})$. Then
$\phi^*(\alpha_X)=\delta_2(\phi)=\alpha_Y$.
\end{theorem}
\begin{proof}This follows from the above commutative diagram.\end{proof}

Let $\Gamma_1$ be a graph, let $\Gamma_2$ be an $\mathrm{Aut}(\Gamma_1)$-stable subgraph, such that in each graph, all vertices has degree at least $4$. We work over a fixed field $k_0$. Let $X_i/k_i$ be universal curve of $\Gamma_i$.

\begin{lemma} There exists a map $t\colon \mathrm{Spec}(k_1)\to\mathrm{Spec}(k_2)$ and a commutative diagram
\[\xymatrix{X_1\ar[d]&(X_2)_{k_1}\ar[l]\ar[r]\ar[d] & X_2\ar[d]\\
\mathrm{Spec}(k_1)&\mathrm{Spec}(k_1)\ar[l]_{\sim}\ar[r]^-t &\mathrm{Spec}(k_2)}\] where the right square is cartesian, the left square is a composition of partial normalization and closed immersion.
\end{lemma}

\begin{proof}Note that we have natural surjective morphism $$\mathrm{Fgt}\colon \mathcal{M}_{\Gamma_1}=\prod_{v\in V_1}\mathcal{M}_{0,E_1(v)}\to\prod_{v\in V_2}\mathcal{M}_{0,E_2(v)}\cong \mathcal{M}_{\Gamma_2}$$ induced by forgetting vertices and forgetting nodes. Note that there exists universal family of curves with marked nodes $\mathcal{Y}_i\to {\mathcal{M}}_{\Gamma_i}$, and there exist a morphism $\mu\colon\mathcal{Y}_2\times_{{\mathcal{M}}_{\Gamma_2}}{\mathcal{M}}_{\Gamma_1}\to\mathcal{Y}_1$ given by partial normalization and closed immersion. It is easy to see $\mathrm{Fgt}$ is $\mathrm{Aut}(\Gamma_1)\times\mathrm{Aut}(\Gamma_2)$-equivariant. Taking quotient we get morphism $\mathcal{D}_{\Gamma_1}\to\mathcal{D}_{\Gamma_2}$. Similarly by equivariance of $\mu$, we get morphism $[\mathcal{Y}_2/\mathrm{Aut}(\Gamma_2)]\times_{\mathcal{D}_{\Gamma_2}}\mathcal{D}_{\Gamma_1}\to[\mathcal{Y}_1/\mathrm{Aut}(\Gamma_1)]$. Note that $[\mathcal{Y}_i/\mathrm{Aut}(\Gamma_i)]\to[{\mathcal{M}}_{\Gamma_i}/\mathrm{Aut}(\Gamma_i)]$ is the universal family $\mathcal{X}_i\to\mathcal{D}_i$, hence
we get the following diagram of universal families:
$$\xymatrix{
\mathcal{X}_1\ar[d] & \mathcal{D}_{\Gamma_1}\times_{\mathcal{D}_{\Gamma_2}}\mathcal{X}_2\ar[l]_-t\ar[r]\ar[d] & \mathcal{X}_2\ar[d] \\%
\mathcal{D}_{\Gamma_1} & \mathcal{D}_{\Gamma_1}\ar[r]\ar[l]_{\sim} & \mathcal{D}_{\Gamma_2}
}
$$
The $X_1/k_1$, $(X_2)_{k_1}/k_1$ and $X_2/k_2$ are the generic fibers of the columns.
\end{proof}
\begin{theorem}\label{thm2}Let $\alpha_i$ be the Brauer class associated with $\Gamma_i$. Then $\alpha_2$ maps to $\alpha_1$, under the natural map induced by base change and restriction: $\mathrm{Br}(\mathrm{Pic}^0_{X_2/k_2})\to\mathrm{Br}(\mathrm{Pic}^0_{(X_2)_{k_1}/k_1})\to\mathrm{Br}(\mathrm{Pic}^0_{X_1/k_1}).$
\end{theorem}
\begin{proof} By functoriality of Leray spectral sequence, we see $\alpha_2$ restricts to $\mathrm{d}_2^{0,1}(1_{\mathrm{Pic}^0_{(X_2)_{k_1}/k_1}})$. Then apply Theorem \ref{thmpull} to $(X_2)_{k_1}\to X_1$, we see $\mathrm{d}_2^{0,1}(1_{\mathrm{Pic}^0_{(X_2)_{k_1}/k_1}})$ maps to $\alpha_1$.
\end{proof} 
\section{Examples}\label{sec6}
We give some examples where the period and index of Brauer class can be determined. In all these examples period to equal index.
\begin{example}\label{Example00} Let $g\geq3$ be an integer. Let $\Gamma$ be the graph with $g-1$ vertices $v_1,\dots, v_{g-1}$, and two edges $e_i,f_i$ between $\{v_i,v_{i+1}\}$ for $i\in\mathbb{Z}/(g-1)\mathbb{Z}$, then $g(\Gamma)=g$ (See Fig. \ref{Example00} for $g=5$). Let $\alpha$ be the Brauer class associated with graph $\Gamma$. Note the $g$ loops $f_1\dots f_{g-1}$,\ $e_if_i$ form a basis for $\mathrm{H}_1(\Gamma,\mathbb{Z})$ as $\mathbb{Z}/(g-1)\mathbb{Z}$-module, with rotation action on indices. Applying Theorem \ref{thm} to the loop $f_1\dots f_{g-1}$, we know $g-1|\mathrm{per}(\alpha)$. By Theorem \ref{ind}, $\mathrm{ind}(\alpha)|g-1$, so $\mathrm{per}(\alpha)=\mathrm{ind}(\alpha)=g-1$.
\end{example}
\begin{remark}Let $g\geq3$ be an integer. Let $\mathcal{M}_g$ be the moduli stack of smooth genus $g$ curves. Let $k$ be its function field, let $X\to k$ be the universal curve. Let $\alpha=\mathrm{d}_2^{0,1}(1_{\mathrm{Pic}^0_{X/k}})$ be the Brauer class. Then since period drops by specialization, the previous example shows $g-1|\mathrm{per}(\alpha)$, thus $\mathrm{per}(\alpha)=\mathrm{ind}(\alpha)=g-1$.
\end{remark}

\begin{tikzpicture}[scale=1]
\coordinate (A1) at (2,1);
\coordinate (A2) at (0,1);
\coordinate (A3) at (0,-1);
\coordinate (A4) at (2,-1);

\draw (A1) .. controls (0.95,1.1) and (1.05,1.1) .. (A2);
\draw (A1) .. controls (0.95,0.9) and (1.05,0.9) .. (A2);
\draw (A2) .. controls (-0.1,-0.05) and (-0.1,0.05) .. (A3);
\draw (A2) .. controls (0.1,-0.05) and (0.1,0.05) .. (A3);
\draw (A3) .. controls (0.95,-1.1) and (1.05,-1.1) .. (A4);
\draw (A3) .. controls (0.95,-0.9) and (1.05,-0.9) .. (A4);
\draw (A4) .. controls (1.9,0.05) and (1.9,-0.05) .. (A1);
\draw (A4) .. controls (2.1,0.05) and (2.1,-0.05) .. (A1);

\node [above] at (A1) {$v_1$};
\node [above] at (A2) {$v_2$};
\node [below] at (A3) {$v_3$};
\node [below] at (A4) {$v_4$};
\node [above] at (1,1) {$e_1$};
\node [left ] at (0,0) {$e_2$};
\node [below] at (1,-1) {$e_3$};
\node [right] at (2,0) {$e_4$};
\node [below] at (1,1) {$f_1$};
\node [right] at (0,0) {$f_2$};
\node [above] at (1,-1) {$f_3$};
\node [left] at (2,0) {$f_4$};
\node [below] at (1,-1.5) {Fig \ref{Example00}};
\end{tikzpicture}
\hspace{2cm}\begin{tikzpicture}[scale=1.3]
\coordinate (A1) at (0.00,1.00);
\coordinate (A2) at (-0.95,0.31);
\coordinate (A3) at (-0.59,-0.81);
\coordinate (A4) at (0.59,-0.81);
\coordinate (A5) at (0.95,0.31);

\draw (A1) -- (A2) -- (A3) -- (A4) -- (A5) -- (A1);
\draw (A1) -- (A3) -- (A5) -- (A2) -- (A4) -- (A1);
\node [above] at (A1) {$v_1$};
\node [left] at (A2) {$v_2$};
\node [below] at (A3) {$v_3$};
\node [below] at (A4) {$v_4$};
\node [right] at (A5) {$v_5$};
\node [below] at (0,-1.2) {Fig \ref{Example1}.i};
\end{tikzpicture}
\hspace{2cm}
\begin{tikzpicture}[scale=1.3]
\coordinate (A1) at (0.00,0.00);
\coordinate (A2) at (0.00,1.00);
\coordinate (A3) at (-1.06,-0.55);
\coordinate (A4) at (0.57,-0.95);
\coordinate (A5) at (1.00,0.00);

\draw (A2) -- (A3) -- (A4) -- (A5) -- (A2);
\draw (A2) -- (A4);
\draw (A1) -- (A2);
\draw (A1) -- (A3);
\draw (A1) -- (A4);
\draw (A1) -- (A5);
\draw (A3) -- (A5);

\node [above left] at (A1) {$v_4$};
\node [left] at (A2) {$v_5$};
\node [below] at (A3) {$v_1$};
\node [below] at (A4) {$v_2$};
\node [right] at (A5) {$v_3$};
\node [below] at (0,-1.2) {Fig \ref{Example1}.ii};
\end{tikzpicture}

\begin{example}\label{Example1}Let $\Gamma=K_5$ be the complete graph with $5$ vertices. Label the vertices of the graph by $v_1,\dots, v_5$, then $\mathrm{Aut}(\Gamma)=S_5$ acts by permutation on vertices. In this case, the loops $v_{i-1}v_iv_{i+1}(i=1,\dots,5), v_1v_2v_3v_4v_5$ form a basis of $H_1(\Gamma,\mathbb{Z})$. Take $\sigma=(12345)\in S_5$, then $L=v_1v_2v_3v_4v_5$ is a cycle class fixed by permutation $\sigma$, and is a direct summand. Let $\alpha$ be the Brauer class associated with $\Gamma$, then by Theorem \ref{thm}, we have $5|\mathrm{per}(\alpha)$.
Note that $g(\Gamma)=6$, therefore by Theorem \ref{ind} we conclude $\mathrm{per}(\alpha)=\mathrm{ind}(\alpha)=5$.

On the other hand, although the loop $L'=v_1v_2v_3$ is invariant under $\sigma=(123)$, it does not generate a summand of $\mathrm{H}_1(\Gamma,\mathbb{Z})$ as $\langle\sigma\rangle$-module, so we cannot conclude $3|\mathrm{per}(\alpha).$
\end{example}

\begin{example}\label{Example4}
Let $\Gamma$ be the graph of complete graph $K_4$, with edge doubled along a simple length four loop (Fig \ref{Example4}.i). Let $\alpha$ be the associated Brauer class. Consider the automorphism $\sigma\in\mathrm{Aut}(\Gamma)\colon v_i\mapsto v_{5-i}$, but switch the edges between $v_1v_4, v_2v_3$ (we may view this automorphism as $180$ degree rotation in Fig \ref{Example4}.ii). Then apply Theorem \ref{thm} to the loop of $v_2v_3$, we can show  $2|\mathrm{per}(\alpha)$. Since $g(\Gamma)=7$, we know $\mathrm{ind}(\alpha)|6$. Apply Theorem \ref{extra1} to $\Gamma_0=\Gamma$, we see $\mathrm{ind}(\alpha)|10$, thus $\mathrm{per}(\alpha)=\mathrm{ind}(\alpha)=2$.
\end{example}

\begin{center}
\begin{tikzpicture}[scale=1]
\coordinate (A1) at (2,1);
\coordinate (A2) at (0,1);
\coordinate (A3) at (0,-1);
\coordinate (A4) at (2,-1);

\draw (A1) .. controls (0.95,1.1) and (1.05,1.1) .. (A2);
\draw (A1) .. controls (0.95,0.9) and (1.05,0.9) .. (A2);
\draw (A2) .. controls (-0.1,-0.05) and (-0.1,0.05) .. (A3);
\draw (A2) .. controls (0.1,-0.05) and (0.1,0.05) .. (A3);
\draw (A3) .. controls (0.95,-1.1) and (1.05,-1.1) .. (A4);
\draw (A3) .. controls (0.95,-0.9) and (1.05,-0.9) .. (A4);
\draw (A4) .. controls (1.9,0.05) and (1.9,-0.05) .. (A1);
\draw (A4) .. controls (2.1,0.05) and (2.1,-0.05) .. (A1);

\draw (A1) -- (A3);\draw (A2) -- (A4);
\node [above] at (A1) {$v_1$};
\node [above] at (A2) {$v_2$};
\node [below] at (A3) {$v_3$};
\node [below] at (A4) {$v_4$};
\node [below] at (1,-1.5) {Fig \ref{Example4}.i};
\end{tikzpicture}\hspace{3cm}
\begin{tikzpicture}[scale=0.9]
\coordinate (A1) at (0,1.5);
\coordinate (A2) at (0,0.5);
\coordinate (A3) at (0,-0.5);
\coordinate (A4) at (0,-1.5);

\draw (A1) .. controls (-0.7,0.5) and (-0.7,-0.5) .. (A4);
\draw (A1) .. controls (0.7,0.5) and (0.7,-0.5) .. (A4);
\draw (A2) .. controls (-0.1,0.1) and (-0.1,-0.1) .. (A3);
\draw (A2) .. controls (0.1,0.5) and (0.1,-0.1) .. (A3);
\draw (A1) .. controls (-0.4,0.6) and (-0.4,0.6) .. (A3);
\draw (A1) .. controls (-0.1,0.7) and (-0.1,0.8) .. (A2);
\draw (A1) .. controls (0.1,0.7) and (0.1,0.8) .. (A2);
\draw (A3) .. controls (-0.1,-0.7) and (-0.1,-0.8) .. (A4);
\draw (A3) .. controls (0.1,-0.7) and (0.1,-0.8) .. (A4);
\draw (A2) .. controls (0.4,-0.4) and (0.4,-0.6) .. (A4);
\node [right] at (A1) {$v_1$};
\node [right] at (A2) {$v_2$};
\node [left] at (A3) {$v_3$};
\node [right] at (A4) {$v_4$};\node [below] at (0,-1.7) {Fig \ref{Example4}.ii};
\end{tikzpicture}\hspace{2cm}
\end{center} 
\begin{example}\label{hybrid}

Let $\Gamma$ be the graph with vertices $v_1,\dots,v_4; w_1,\dots,w_4$, with edges as in the following diagram. This graph admits Fig \ref{Example00} as an $\mathrm{Aut}(\Gamma)$-invariant subgraph. Apply Theorem \ref{thm} and Theorem \ref{thm2}, we know $\mathrm{per}(\alpha)=4$. Note that index drops by base change\footnote{Azumaya algebras pull back to Azumaya algebras.}, we conclude that $\mathrm{ind}(\alpha)=4$.
\end{example}

\begin{example}\label{splitcurve}Let $\Gamma=K_{3,4}$ be the complete bipartite graph. Let $\alpha$ be the associated Brauer class. Since $\mathrm{g}(\Gamma)=6$, we know $\mathrm{per}(\alpha)|g-1=5$. On the other hand, by \ref{simpleest}, $\mathrm{per}(\alpha)$ divides $|\mathrm{Aut}(K_{3,4})|=|S_3\times S_4|=4!\times 3!$. Then $\mathrm{per}(\alpha)|(5,4!\times 3!)=1$, hence the Brauer class is trivial.
\end{example}

\begin{center}
\begin{tikzpicture}[scale=1]
\coordinate (A1) at (2,1);
\coordinate (A2) at (0,1);
\coordinate (A3) at (0,-1);
\coordinate (A4) at (2,-1);
\coordinate (B1) at (1,0.3);
\coordinate (B2) at (1,-0.3);
\coordinate (B3) at (0.7,0);
\coordinate (B4) at (1.3,0);
\draw (A2).. controls (0.48,0.76) and (0.61,0.67) ..(B1);
\draw (A2).. controls (0.39,0.63) and (0.52,0.54) ..(B1);
\draw (A2).. controls (0.37,0.61) and (0.46,0.48) ..(B3);
\draw (A2).. controls (0.33,0.39) and (0.24,0.52) ..(B3);

\draw (A3).. controls (0.37,-0.61) and (0.46,-0.48) ..(B3);
\draw (A3).. controls (0.33,-0.39) and (0.24,-0.52) ..(B3);
\draw (A3).. controls (0.48,-0.76) and (0.61,-0.67) ..(B2);
\draw (A3).. controls (0.52,-0.54) and (0.39,-0.63) ..(B2);
\draw (A4).. controls (1.48,-0.54) and (1.61,-0.63) ..(B2);
\draw (A4).. controls (1.39,-0.67) and (1.52,-0.76) ..(B2);
\draw (A4).. controls (1.76,-0.52) and (1.67,-0.39) ..(B4);
\draw (A4).. controls (1.54,-0.48) and (1.63,-0.61) ..(B4);
\draw (A1).. controls (1.76,0.52) and (1.67,0.39) ..(B4);
\draw (A1).. controls (1.54,0.48) and (1.63,0.61) ..(B4);
\draw (A1).. controls (1.48,0.54) and (1.61,0.63) ..(B1);
\draw (A1).. controls (1.39,0.67) and (1.52,0.76) ..(B1);

\draw (A1) .. controls (0.95,1.1) and (1.05,1.1) .. (A2);
\draw (A1) .. controls (0.95,0.9) and (1.05,0.9) .. (A2);
\draw (A2) .. controls (-0.1,-0.05) and (-0.1,0.05) .. (A3);
\draw (A2) .. controls (0.1,-0.05) and (0.1,0.05) .. (A3);
\draw (A3) .. controls (0.95,-1.1) and (1.05,-1.1) .. (A4);
\draw (A3) .. controls (0.95,-0.9) and (1.05,-0.9) .. (A4);
\draw (A4) .. controls (1.9,0.05) and (1.9,-0.05) .. (A1);
\draw (A4) .. controls (2.1,0.05) and (2.1,-0.05) .. (A1);

\node [above] at (A1) {$v_1$};
\node [above] at (A2) {$v_2$};
\node [below] at (A3) {$v_3$};
\node [below] at (A4) {$v_4$};
\node [above] at (B1) {$w_1$};
\node [left] at (B3) {$w_2$};
\node [below] at (B2) {$w_3$};
\node [right] at (B4) {$w_4$};
\node [below] at (1,-1.5) {Fig \ref{hybrid}.i};
\end{tikzpicture}\hspace{3cm}
\begin{tikzpicture}[scale=1]
\coordinate (v1) at (-1,1);
\coordinate (v2) at (-1,0);
\coordinate (v3) at (-1,-1);
\coordinate (w1) at (1,1);
\coordinate (w2) at (1,0.3);
\coordinate (w3) at (1,-0.3);
\coordinate (w4) at (1,-1);%
\draw (v1)--(w1);\draw (v1)--(w2);\draw (v1)--(w3);\draw (v1)--(w4);
\draw (v2)--(w1);\draw (v2)--(w2);\draw (v2)--(w3);\draw (v2)--(w4);
\draw (v3)--(w1);\draw (v3)--(w2);\draw (v3)--(w3);\draw (v3)--(w4);
\node [left] at (v1) {$v_1$};
\node [left] at (v2) {$v_2$};
\node [left] at (v3) {$v_3$};
\node [right] at (w1) {$w_1$};
\node [right] at (w2) {$w_2$};
\node [right] at (w3) {$w_3$};
\node [right] at (w4) {$w_4$};
\node [below] at (0,-1.2) {Fig \ref{splitcurve}};
\end{tikzpicture}\hspace{1.6cm}
\end{center}

\section{Further Discussion}\label{sec5}
\subsection{Period and index}
Let $k_0$ be a field. Let $\Gamma$ be a graph, such that each vertex has degree at least $4$. Let $X/k$ be the associated curve and $\alpha$ be the associated Brauer class. The period can be calculated by group cohomology, thus $\mathrm{per}(\alpha)$ is independent of choice $k_0$. We may ask
\begin{question}Does $\mathrm{ind}(\alpha)$ depend on the field $k_0$?
\end{question}
In all the examples above, as long as we can determine period and index, $\mathrm{per}(\alpha)=\mathrm{ind}(\alpha)$. We may ask if this is true in general.
\begin{question} Is it always true that $\mathrm{per}(\alpha)=\mathrm{ind}(\alpha)$?
\end{question}
Here is an example where the author cannot determine the period and index.
\begin{example}\label{Example3} Let $\Gamma$ be the graph of truncated icosahedron(soccer), with adjacent edges of hexagons doubled. Apply Theorem \ref{thm} to the hexagon containing $v_1,v_2,v_3$, we know $3|\mathrm{per}(\alpha)$. Apply Theorem \ref{thm} to the loop $v_2v_4$, we know $2|\mathrm{per}(\alpha)$. Apply Theorem \ref{thm} to the loop $v_1v_4v_5v_6v_7$, one know $5|\mathrm{per}(\alpha)$. Thus $30|\mathrm{per}(\alpha)$. In this case $g=61$, so the possible period and index could be $(30,30),(30,60)$ or $(60,60)$.
\end{example}

\begin{center}
\begin{tikzpicture}[scale=0.7]
\coordinate (A1) at (0.00,1.00);
\coordinate (A2) at (-0.95,0.31);
\coordinate (A3) at (-0.59,-0.81);
\coordinate (A4) at (0.59,-0.81);
\coordinate (A5) at (0.95,0.31);
\coordinate (B1) at (0.00,1.61);
\coordinate (B2) at (-1.53,0.50);
\coordinate (B3) at (-0.95,-1.30);
\coordinate (B4) at (0.95,-1.30);
\coordinate (B5) at (1.53,0.50);
\coordinate (C1) at (0.67,1.78);
\coordinate (C2) at (-0.67,1.78);
\coordinate (C3) at (-1.48,1.18);
\coordinate (C4) at (-1.89,-0.09);
\coordinate (C5) at (-1.58,-1.04);
\coordinate (C6) at (-0.50,-1.83);
\coordinate (C7) at (0.50,-1.83);
\coordinate (C8) at (1.58,-1.04);
\coordinate (C9) at (1.89,-0.09);
\coordinate (C0) at (1.48,1.18);

\draw (A1) -- (A2) -- (A3) -- (A4) -- (A5) -- (A1);
\draw (C1) -- (C2);
\draw [ultra thick](C2)--(C3);
\draw (C3) -- (C4);
\draw [ultra thick](C4)--(C5);
\draw (C5) -- (C6);\draw [ultra thick](C6)--(C7);
\draw (C7) -- (C8);\draw [ultra thick](C8)--(C9);
\draw (C9) -- (C0);
\draw [ultra thick](C0) -- (C1);
\draw [ultra thick](A1) -- (B1);
\draw [ultra thick](A2) -- (B2);
\draw [ultra thick](A3) -- (B3);
\draw [ultra thick](A4) -- (B4);
\draw [ultra thick](A5) -- (B5);
\draw (B1) -- (C1);\draw (B1) -- (C2);
\draw (B2) -- (C3);\draw (B2) -- (C4);
\draw (B3) -- (C5);\draw (B3) -- (C6);
\draw (B4) -- (C7);\draw (B4) -- (C8);
\draw (B5) -- (C9);\draw (B5) -- (C0);

\node [below] at (A1) {$v_1$};
\node [left] at (B2) {$v_2$};
\node [above] at (C2) {$v_3$};
\node [right] at (A2) {$v_4$};
\node [above right] at (A3) {$v_5$};
\node [above right] at (A4) {$v_6$};
\node [above] at (A5) {$v_7$};
\node [below] at (0,-2.2) {Fig \ref{Example3}};
\end{tikzpicture}
\end{center} 

\subsection{Cyclicity}
Finally, we say a little bit about cyclicity.
\begin{theorem} If $g$ is even, then the Brauer class in Theorem \ref{Example00} is cyclic.
\end{theorem}
\begin{proof} Let $k'$ be the splitting field of $X$ as defined in section \ref{2.3}. Note that $\mathrm{Gal}(k'/k)=\mathrm{Aut}(\mathrm{\Gamma}_g)=D_{g-1}\times(\mathbb{Z}/2\mathbb{Z})^{g-1}$. Take the invariant subfield $k''$ of $(\mathbb{Z}/2\mathbb{Z})^{g-1}$. Then $k'/k''$ is an extension of degree $2^{g-1}$ and $k''/k$ is a dihedral extension. The Brauer class $\alpha\in\mathrm{Br}(k)$ restricts to $0\in\mathrm{Br}(k')$, since the universal curve splits over $k'$. This implies the class $\mathrm{Res}_{k''/k}(\alpha)=0$ in $\mathrm{Br}(k'')$, as we know $\alpha$ has odd degree $g-1$ and $0=\mathrm{Cores}_{k'/k''}\circ\mathrm{Res}_{k'/k''}(\mathrm{Res}_{k''/k}\alpha)=2^{g-1}\mathrm{Res}_{k''/k}(\alpha)$. Then we know the Brauer class is cyclic since dihedral algebras are known to be cyclic, see \cite[1]{rowen}.
\end{proof}

The cyclicity of the last example maybe expected from the symmetries of the graph. In Example \ref{Example1}, the automorphism graph is symmetric group $S_5$. In Example \ref{Example3}, the graph has automorphism $A_5$ has a normal subgroup, we may ask:
\begin{question}Are the Brauer classes in Example \ref{Example1} or Example \ref{Example3} cyclic?
\end{question}

\bibliographystyle{alpha}
\bibliography{references}

\end{document}